\documentclass[reqno,makeidx,11pt]{amsart}
\usepackage{amssymb,amsfonts,amsmath,graphicx}

\usepackage{pgf,tikz}
\usepackage{mathrsfs}
\usetikzlibrary{arrows}
\usepackage{amssymb,fancyhdr,txfonts,pxfonts,dsfont, fixmath}

\definecolor{qqqqff}{rgb}{0.,0.,1.}
\definecolor{ffqqqq}{rgb}{1.,0.,0.}

\usepackage{etoolbox}

\makeatletter
\patchcmd{\@settitle}{\uppercasenonmath\@title}{}{}{}
\patchcmd{\@setauthors}{\MakeUppercase}{}{}{}
\makeatother

\def\F{{\mathbb{F}}}

\def\N{{\mathbb{N}}}
\def\0{{\mathbb{O}}}

\def\Q{{\mathbb{Q}}}
\def\R{{\mathbb{R}}}

\def\Z{{\mathbb{Z}}}

\def\t{{\mathbbb{T}}}

\def\cA{{\mathcal A}}

\def\cC{{\mathcal C}}

\def\cK{{\mathcal K}}

\def\cT{{\mathcal T}}

\def\Aut{{\rm Aut \, }}

\def \Iso {\hbox{Iso}}

\def \Map {\hbox{Map}}
\def\Sym {\hbox{Sym}}

\newcommand{\norm}[1]{{\left\|{#1}\right\|}}

\newcommand{\abs}[1]{{\left|{#1}\right|}}
\newcommand{\scal}[1]{{\left\langle{#1}\right\rangle}}
\newcommand{\set}[1]{{\left\{{#1}\right\}}}

\newcommand{\actson}{\curvearrowright}

\def\fig{ \centerline{Fig. \the\count200\global\advance\count200 by 1}}
\newtheorem{thm}{Theorem}[section]
\newtheorem{cor}[thm]{Corollary}
\newtheorem{lem}[thm]{Lemma}
\newtheorem{prop}[thm]{Proposition}

\newtheorem*{thm*}{Theorem}
\newtheorem*{lem*}{Lemma}
\newtheorem*{cor*}{Corollary}

\theoremstyle{definition}
\newtheorem{defn}[thm]{Definition}
\newtheorem{ex}[thm]{Example}
\newtheorem{exs}[thm]{Examples}

\theoremstyle{remark}
\newtheorem{rem}[thm]{Remark}

\title[Amenable actions preserving a locally finite metric]
{Amenable actions preserving a locally finite metric}
\author{Claire Anantharaman-Delaroche}
\address{Institut Denis Poisson,
Universit\'e d'Orl\'eans, Universit\'e de Tours, CNRS,
\newline
\indent B. P. 6759, 45067 Orl\'eans Cedex 2, France}
\email{claire.anantharaman@univ-orleans.fr}



\begin{document}

\begin{abstract} 
The class $\cA$ of countable groups that admit a faithful, transitive, amenable -- in the sense that there is an invariant mean -- action on a set has been widely investigated in the past. In this paper, we no longer require the action to be transitive, but we ask for it to preserve a locally finite metric (and still to be faithful and amenable). The groups having such actions are those that embed into a totally disconnected amenable locally compact group. Then we focus on the subclass $\cA_{1}$ of groups for which the actions are moreover transitive. This class is strictly contained into $\cA$ and includes non-amenable groups. An important particular case of actions preserving a locally finite metric is given by actions by automorphisms of locally finite connected graphs. We take this opportunity, in our partly expository paper, to review some nice results about amenable actions in this setting.

\end{abstract}
\maketitle

\section{Introduction}
  
Let $G$ be a  group acting (to the left) on a  set $X$.\footnote{In this paper, $G$ and $X$ will always be endowed with the discrete topology.} It is an old and important problem to understand how some properties of $G$ are reflected into properties of the action and to investigate the converse, namely the influence on $G$ of some properties of the action.

The origin of the problem that interests us here goes back to the very beginning of the previous century, in the desire to assimilate the notions of measure and integral. Since it had been shown by Vitali, as early as 1905, that the Lebesgue measure on $\R^n$ could not be extended as a {\it countably additive} measure on all the subsets of $\R^n$, still invariant under the action  by isometries of $\R^n$, a natural question was to see whether such an extension exists when only requiring its {\it finite additivity}.  This question was answered negatively by Hausdorff (1914) when $n\geq 3$ and positively by Banach  (1923) for $n=1, 2$. The crucial contribution of Hausdorff, improved later by Banach-Tarski (1924) is that, due to the fact that the group of rotations $SO(3,\R)$ contains the free group $\F_2$ of rank $2$ as subgroup, there does not exist a finitely additive measure on the unit sphere $\mathbb S^2$ of $\R^3$,  invariant under the $SO(3,\R)$-action. 

In order to avoid the confusion between countably and finitely additive measures, a finitely additive measure will rather be called a {\it mean} in this article (see Section \ref{sect:glimpse}).

In its influential paper \cite{vN29}, von Neumann discovered that the actual reason for the existence of a mean on $X =\R$ and on $X=\R^2$, invariant under the action of their isometry groups, was what is now called the amenability of these groups. He also showed that every group containing $\F_2$ as a subgroup is non-amenable.

In this text, an action $G\actson X$ is said to be amenable if there exists a $G$-invariant mean on $X$ (see Section \ref{subsect:amen}). Since there is sometimes a confusion between several notions of amenable action,  we give a short account on the subject in Section \ref{sect:glimpse} for the reader's convenience. 

 It is easily seen that any action of an amenable group is amenable and that, conversely, if a {\it free} action of $G$ is amenable then $G$ has to be amenable. In \cite[Problem, page 18]{Green} Greenleaf asked whether this converse holds for any ``reasonable'' action of $G$. Of course this action should be faithful, otherwise one would be investigating the action of a quotient of $G$. Secondly, the action should be transitive, since it is easy to construct  faithful amenable actions of any group $G$, for instance the action on $X= G\sqcup \{x_0\}$ where $x_0$ is a fixed point, the action on $G$ being by left translations. Therefore the natural question is the following one:

\begin{itemize}
\item[]{\it Let $\cA$ be the class of all   countable\footnote{Throughout this paper, countable means {\it infinite} countable.} groups which admit a faithful, transitive, amenable action. Is $\cA$ reduced to the  class of amenable groups?}
\end{itemize}

This question was answered negatively by van Douwen  in a posthumous  paper \cite{vanDou} published in 1990, which shows that the free group $\F_2$ with two generators is in $\cA$.  Even more,  the action of $\F_2$ can be constructed to be almost free in the sense that every element of $\F_2\setminus \set{e}$ has at most a finite number of fixed points. 

A detailed study of the class $\cA$ was carried out by Glasner and Monod in \cite{GM}. An obvious obstruction for a countable group to be in $\cA$ is the Kazhdan property (T). Among the results contained in \cite{GM}, one finds a complete characterization of the free products that are in $\cA$. In particular, the class $\cA$ is stable under free products. Other methods of constructing groups in the class $\cA$ have been  presented  by Grigorchuk and Nekrashevych \cite{GN}.  Later on, much more examples of groups in the class $\cA$ were given \cite{Fima, Moon10, Moon11, Moon11bis}.

We observe that non-necessarily transitive  actions, but with finite stabilizers and more generally with amenable stabilizers, are amenable if and only if the acting group is amenable (Lemma \ref{lem:useful}). There is a   class of group actions $G\actson X$, more general than actions with finite stabilizers,  which shows certain rigidity features and seems of interest, namely the actions such that the stabilizers of every point have  finite orbits. We call them {\it f.s.o.} ({\it finite stabilizers orbits}) {\it actions}.  Whenever an action is f.s.o., the closure $G'$ of the range of $G$ in the space $\Map(X)$ of maps from $X$ to $X$, equipped with the topology of pointwise convergence, is a locally compact group (of course Hausdorff and totally disconnected), acting properly on $X$  (Theorem \ref{thm:exist_proper}). We call $G'$ the {\it Schlichting group} associated with the action. F.s.o. actions on a countable set are exactly the same as the actions preserving  a locally finite metric \cite[Proposition 3.3]{AD13}. 

  The  stabilizers of a  f.s.o.~action  $G\actson X$ form a family of almost normal subgroups of $G$, which are mutually commensurate (Definition \ref{defn:com}). For such an action we show that the following conditions are equivalent (Proposition \ref{co_sco3}):
 {\it  \begin{itemize}
 \item[(i)] the global action $G\actson X$ is amenable;
 \item[(ii)] there exists a $G$-orbit $Y$ such that the action $G\actson Y$ is amenable;
 \item[(iii)] the  action of $G$ on each of its orbits is amenable;
 \item[(iv)] the Schlichting group $G'$ is amenable.
 \end{itemize}}
 
 This is a rigidity property of f.s.o. actions. Indeed,  there exist amenable (not f.s.o.) actions without any orbit on which the group acts amenably.
 
 For a transitive f.s.o. action, that is, an action of the form $G\actson G/H$ where $H$ is an almost normal subgroup of $G$, the group $G'$ was considered by Schlichting in \cite{Sch79,Sch80}. In this case, the equivalence between the  amenability of the action and the amenability of $G'$ had been established in \cite[Proposition 5.1]{Tza}.  
  
  When the stabilizers of a faithful action $G\actson X$ are finite we have $G=G'$ and, as already said, the action is amenable if and only if $G$ is amenable. When the stabilizers have a finite index in $G$, the action is of course f.s.o. and amenable, and the group $G'$ is compact. Thus the amenability of $G'$ does not imply the amenability of $G$.  It suffices to consider {\it any} residually finite group $G$ and any family $(H_i)_{i\in I}$ of finite index subgroups $H_i$ such that $\bigcap_{g\in G, i\in I} gH_ig^{-1} = \{e\}$. The obvious $G$-action on $X=\sqcup_{i\in I} G/H_i$ is faithful, amenable, f.s.o., and $G'$ is compact, whereas $G$ can be a non-abelian free group, or a Kazhdan group for instance.
 
 As a consequence of Proposition \ref{co_sco3}  we see that if $G\actson X$ is an amenable f.s.o. action, then every subgroup $H$ of $G$ acts amenably on all its $H$-orbits. This was already observed in \cite{AD13} for transitive $G$-actions. This fact is false without the f.s.o. assumption. Indeed, in \cite{MoPo, Pes} one finds examples of amenable transitive actions $G\actson G/K$ for which there exist intermediate subgroups $H$ ({\it i.e.}, $K < H < G$) such that the action $H\actson H/K$ is not amenable.

Therefore,   amenable f.s.o.~actions seem to behave more nicely and it is tempting to ask whether this property of the action implies the amenability of the group when moreover the action is faithful and transitive.  So it appears to be interesting to consider the class  $\cA_{1}\subset \cA$ of countable groups that admit faithful, transitive, amenable f.s.o.~actions (equivalently, actions preserving a locally finite metric) and to compare it to the class of amenable groups and to $\cA$. This is the main purpose of this paper.

The reminders of Section \ref{sect:glimpse} on amenable actions are followed  in Section \ref{sect:f.s.o.}
 by some prelimi\-naries on f.s.o. actions. They are related to   the different ways to embed densely a (discrete) group in a totally disconnected locally compact ({\it t.d.l.c.} in short) group, about which we go into some  details. An important particular case of f.s.o. actions are those by automorphisms of locally finite connected graphs, viewed as actions on their set of vertices. In fact, they are  the most general examples of f.s.o. actions of finitely generated groups on countable sets (Proposition \ref{prop:graph}).

 In Section \ref{sect:amenact} we give   equivalent characterizations of amenable f.s.o. actions. In particular we show that  the class $\widetilde{\cA}_{1}$ of countable   groups having faithful, amenable f.s.o. actions   is the same as the class of countable groups that are subgroups of amenable t.d.l.c. groups. Thus, there are some  obstructions for a group to be in $\widetilde{\cA}_{1}$. For instance, every Property (T) subgroup of $G \in \widetilde{\cA}_{1}$ must be residually finite. We also review several nice facts about amenable actions on locally finite connected graphs, a subject widely studied in the '90s (Theorem \ref{Woess}) and related to results in \cite{Ab77}. 
  
In section \ref{sect:transitive} we establish some features and properties of the class $\cA_{1}$ (contained both in $\cA$ and in $\widetilde{\cA}_{1}$) with a focus on faithful, transitive, amenable actions on locally finite connected graphs.

We observe in particular that the class $\cA_{1}$ is strictly smaller than $\cA$. The reason is that every countable group can be embedded as a subgroup of a group that belongs to $\cA$, as shown in \cite{GM}, while some groups cannot be embedded in a group belonging to $\cA_{1}$ ({\it e.g} non-residually finite groups with Kazhdan Property (T)).

  We also show that $\cA_{1}$ is strictly bigger than the class of amenable groups since there are groups in $\cA_{1}$ having non-abelian free groups or Property (T) groups as subgroups.  These exam\-ples are obtained as groups acting faithfully, transitively and amenably by automorphisms on the set of vertices of locally finite trees. This latter class of groups consists in ascending HNN-extensions (Propositions \ref{prop:HNN0} and \ref{HNN1}).
    
 {\em In this text all topological spaces are Hausdorff. Graphs are undirected, connected, locally finite, infinite and are  without multiple edges and loops. The notation $\mathfrak G$ will be reserved  for locally compact groups non-necessarily discrete.}

  \section{A short overview on different notions of amenable actions}  \label{sect:glimpse}
  
  A {\it mean} (or state)  on the algebras of functions considered below is a unital positive linear functional. In the particular case of a mean $m$ on $\ell^\infty(X)$, we may identify $m$ with the map  $\mu$ from the collection of subsets of $X$ to $[0,1]$, defined by  $\mu(E) = m(\chi_E)$, where $\chi_E$ is the characteristic function of $E$.  Such a map $\mu$ satisfies the following properties:
  \begin{itemize}
\item[(i)] $(${\it finite additivity}$)$ $\mu(E\cup F) = \mu(E) + \mu(F)$ when $E\cap F = \emptyset$;
\item[(ii)] $\mu(X) = 1$.
\end{itemize}
We say that $\mu$ is a {\it mean on $X$}.

\subsection{Amenability in Greenleaf's sense}\label{subsect:amen}Let $\mathfrak G$ be a locally compact group with a (jointly) continuous action $\mathfrak G \times Z \to Z$ on a locally compact space. We denote by $C_b(Z)$ the $C^*$-algebra of complex-valued bounded continuous functions on $Z$ and by $C_{b}(Z)^c$ its subalgebra of functions $f$ such that $g\mapsto _{g}\!\!f$ is norm-continuous on $\mathfrak G$, where $_gf(z) = f(g^{-1}z)$. A mean $m$ on $A=C_b(Z)$, or on $A=C_{b}(Z)^c$, is said to be {\it invariant} if $m(_gf) = m(f)$ for every $g\in \mathfrak G$ and every $f\in A$.

\begin{defn}\label{def:green}  Following Greenleaf \cite{Greenleaf}, we say that the action $\mathfrak G\actson Z$ is {\it topologically amenable} 
 if there exists an invariant mean $m$ on $C_b(Z)$. This is equivalent to the existence of an invariant mean on $C_{b}(Z)^c$ (see \cite[Theorem 3.1]{Greenleaf}).  

If $\nu$ is a  measure on $Z$ that is quasi-invariant under the action of $\mathfrak G$, then we say that the action on $(Z,\nu)$ is {\it amenable} (in Greenleaf's sense) if there exists an invariant mean on $L^\infty(Z,\nu)$. This is equivalent to the existence of an invariant mean on  $L^\infty(Z,\nu)^c$ (the algebra of functions $f\in L^\infty(Z,\nu)$ such that $g\mapsto _{g}\!\!f$ is norm-continuous). In \cite[page 344]{AL} an amenable action in this sense is called a {\it co-amenable} action.
\end{defn}

Note that  an amenable  action on $(Z,\nu)$ is  topologically amenable, but the converse does not hold in general (see \cite[Example 1]{Greenleaf}). 

Of course, when $Z$ is a discrete set and $\nu$ is the counting measure, the two notions of amenable action are the same since $C_b(Z) = \ell^\infty(Z) = \ell^\infty(Z,\nu)$,  and are equivalent to the existence of an invariant mean on $Z$,  {\it i.e.} a mean $\mu$ such that $\mu(gE) = \mu(E)$ for all $g\in G$ and $E\subset X$.

 Another important particular case concerns the action of $\mathfrak G$ on $Z =\mathfrak G/ H$, where $ H$ is a closed subgroup of $\mathfrak G$. Let $\nu$ be the (essentially unique) quasi-invariant measure on $\mathfrak G/ H$. Then here, the topological action $\mathfrak G \actson  Z$ is amenable\footnote{In this case one also says that $ H$ is {\it co-amenable} in $\mathfrak G$, or that the {\it coset space $\mathfrak G/ H$ is amenable}.} if and only if the action $\mathfrak G \actson (Z,\nu)$ is amenable \cite[Theorem 3.3]{Greenleaf}.  When $ H$ is a normal subgroup this is equivalent to the amenability of the quotient group. If this holds with $ H = \set{e}$ we say that $\mathfrak G$ is {\it amenable}.

Every action of an amenable group is amenable (in the topological sense and on $(Z,\nu)$). Let us also mention that there  are other equivalent definitions of the above notions of amenability (for details we refer to \cite{Greenleaf} and \cite{Eym}). We will only need the following one, that we describe under the assumption that $Z$ is discrete. 

\begin{prop}\label{prop:guiv} Let $\mathfrak G\actson Z$  be a continuous action of a locally compact group on a discrete set $Z$ and let $\pi$ be the unitary (Koopman) representation of $\mathfrak G$ into $\ell^2(Z)$ defined by $\pi(g)(\xi) = \,_{g}\xi$. Then $\mathfrak G \actson Z$ is amenable if and only if there exists a net $(\xi_i)$ of unit vectors in $\ell^2(Z)$ such that $\lim_i \norm{\pi(g)\xi_i - \xi_i}_2 = 0$ uniformly on compact subsets of $\mathfrak G$ (in other terms, if and only if the trivial representation of $\mathfrak G$ is weakly contained into $\pi$).
\end{prop}

Let us also recall for the reader's convenience a well-known result.

\begin{lem}\label{lem:useful}  Let $\mathfrak G\actson Z$  be a continuous action of a locally compact group on a discrete set $Z$. We assume that the action is amenable and that its stabilizers are amenable groups. Then $\mathfrak G$ is amenable.
\end{lem}
 \vspace{-0.4cm}
 
\begin{proof} The trivial representation $\iota_{\mathfrak G}$ of $\mathfrak G$ is weakly contained into the Koopman representation $\pi$. 
We write $(\pi,\ell^2(Z))$ as a Hilbert direct sum of quasi-regular representations $(\lambda_{\mathfrak G/H_i},\ell^2(\mathfrak G/H_i))$ where the $H_i$'s are stabilizers of the action. Since these groups are amenable, the quasi-regular representations $\lambda_{\mathfrak G/H_i}$ are weakly contained into the regular representation $\lambda_{\mathfrak G}$. Therefore $\iota_{\mathfrak G}$ is weakly contained into a multiple of $\lambda_{\mathfrak G}$, hence into $\lambda_{\mathfrak G}$, and so, $\mathfrak G$ is amenable.
\end{proof}

 \subsection{Amenability in Zimmer's sense}  The original definition of Zimmer was given by a suitable  fixed point property (see \cite[Chapter 4]{Zim} for details). We give below a definition analogous to the existence of an  invariant mean. The equivalence with Zimmer's original definition was proved in \cite{Zim77} when $\mathfrak G$ is discrete and in \cite{AEG} for a general second countable locally compact group.
 
 Let $\nu$ be a  measure on $Z$, quasi-invariant under the action of $\mathfrak G$, as in Definition \ref{def:green}. We denote by $\lambda$ the left Haar measure on $\mathfrak G$. We identify canonically $L^\infty(Z,\nu)$ with a subalgebra of $L^\infty(Z\times \mathfrak G, \nu \otimes \lambda)$.
 
\begin{defn} We say that the action $\mathfrak G\actson (Z,\nu)$ is {\it amenable in Zimmer's sense}  if there exists a norm-one projection $m : L^\infty(Z\times \mathfrak G, \nu \otimes \lambda) \to L^\infty(Z,\nu)$ that is $\mathfrak G$-equivariant.
\end{defn}

Such a norm-one projection is a positive linear map. The   $\mathfrak G$-equivariance means that $m(_gF) = \,_{g}m(F)$ for $g\in \mathfrak G$ and $F \in L^\infty(Z\times \mathfrak G, \nu \otimes \lambda)$, where $\mathfrak G$ acts diagonally on $Z\times \mathfrak G$.

The group $\mathfrak G$ is amenable if and only if there exists an action  $\mathfrak G\actson (Z,\nu)$ that is amenable both in the sense of Greenleaf and in the sense of Zimmer, and then any  action of the form $\mathfrak G\actson (Z,\nu)$ is amenable in both  senses.

When $Z = \mathfrak G/H$ as in Section \ref{subsect:amen}, then $\mathfrak G \actson (\mathfrak G/H,\nu)$ is amenable in Zimmer's sense if and only if $H$ is an amenable group.

Amenability has been studied in \cite{A-DR} in various more general other settings. {\it In the rest of the text we will only consider, without further mention, amenable actions in Greenleaf's sense on discrete sets}.

 \section{F.s.o. actions and t.d.l.c. completions}\label{sect:f.s.o.}
\subsection{Isometry groups of locally finite metric spaces} Let us begin by some notation and reminders.
Let $X$ be a set. We denote by $\Map(X)$ the space of maps from $X$ to $X$ and by $\Sym(X)$ its subset of bijections.  We endow $\Map(X)$ with the topology of  pointwise convergence. Equipped with the induced topology, $\Sym(X)$ is a  topological group which acts continuously on $X$. When $X$ is countable, $\Sym(X)$ is a Polish group. However $\Sym(X)$ is not closed into $\Map(X)$ when $X$ is infinite and it is not locally compact. Given a metric $d$ on   $X$ we denote by $\Iso(X,d)$  the subgroup of isometries of $(X,d)$. Whereas $\Iso(X,d) = \Sym(X,d)$ when $d$ is such that $d(x,y) = 1$ if $x\not= y$, the group $\Iso(X,d)$ has nice properties when $d$ is a {\it locally finite metric}, meaning that its balls are finite. In fact there is a more general result (see \cite[Proposition 5.B.5]{CH}) on a metric space $(Z,d)$ where  the  closed balls are only assumed to be compact ({\it i.e.} the metric is proper).

\begin{thm} \label{thm:proper} Let $(Z,d)$ be proper metric space. Then the group $\Iso(Z,d)$ is a second countable locally compact group that acts properly on $Z$.
\end{thm}

When a metric $d$ on $X$ is locally finite, the group $\Iso(X,d)$ is, in addition, totally disconnected, as a consequence of the following easy lemma.

\begin{lem}\label{lem:top} Let $\mathfrak G$ be a locally compact group acting\footnote{For us, group actions will implicitly be   continuous.}  properly and faithfully on a discrete topological space $X$. The stabilizers $\mathfrak G_x = \set{g\in \mathfrak G:gx=x}$ are open and compact. The set of finite intersections of these stabilizers form a basis of neighborhoods of the unit $e$. The topology of $\mathfrak G$ is the topology of pointwise convergence. In particular, $\mathfrak G$ is totally disconnected. If $X$ is countable, then $\mathfrak G$ is metrizable.
\end{lem}

\begin{proof}  Since the action is proper on a discrete set $X$, the first part of the assertion is immediate.  The topology of pointwise convergence is obviously weaker than the initial topology. It remains to prove that if $(g_i)$ is a net in $\mathfrak G$ which converges pointwise to $e$, then the convergence also holds for the initial topology. Let $V= \cap_{1\leq k\leq n} \mathfrak G_{x_k}$ be a neighborhood of $e$ for this topology. Since $\lim_{i} g_i x_k = x_k$ for $1\leq k \leq n$, there exists $i_0$ such that for $i\geq i_0$ we have $g_i x_k = x_k$ for $1\leq k \leq n$, that is $g_i\in V$.

When $X$ is countable the group $\mathfrak G$ is first countable, hence metrisable (see \cite{Bir,Kak}).
\end{proof}

\begin{ex}\label{ex:graph} Let $\Gamma = (X,E)$ be a  locally finite connected graph. Here $X$ denotes the set of vertices  and $E$ is the set of unoriented edges. Locally finite means that each vertex has only a finite number of neighbors. The geodesic  distance $d$ on $X$ is  locally finite and $\Iso(X,d)$ is the group $\Aut(\Gamma)$ of automorphisms of $\Gamma$, that is the subgroup of elements $f\in \Sym(X)$ such that $(x,y)$ is an edge if and only if $(f(x), f(y))$ is an edge. 

An action of a group on $\Gamma$ is a homomorphism from this group into $\Aut(\Gamma)$. By definition it is transitive, or amenable, or faithful, if the corresponding action on $X$ has this property.
\end{ex}

\subsection{The Schlichting completion associated with a f.s.o. action} As said in the introduction, an action $G\actson X$ is said to be f.s.o. if  the orbits of every stabilizer of the action are finite. For $x\in X$, we denote by $G_x$ its stabilizer.
 
 \begin{rem} For $x,y\in X$, the map $g\in G_x \mapsto gy$ induces a bijection from $G_x/(G_x\cap G_y)$ onto the orbit $G_x y$.  It follows that the  action $G\actson X$ is f.s.o. if and only if its stabilizers $G_x$  form a family of mutually commensurate and almost normal  subgroups of $G$,  
 two notions whose definition is recalled below.
\end{rem}

\begin{defn}\label{defn:com} Let $G$ be a  group.
\begin{itemize}
\item[(i)] Let $H, K$ be two subgroups of $G$. We say that $H$ and $K$ are {\it commensurate} if $H\cap K$ is a subgroup of finite index in $H$ and $K$.
\item[(ii)] We say that a subgroup  $H$ of $G$ is {\it almost normal}, or that $(G,H)$ is a {\it Hecke pair}, or that $H$ is {\it commensurated in} $G$, if for every $g\in G$ the subgroups $H$ and $gHg^{-1}$ are commensurate.
\end{itemize}
\end{defn}

Obvious examples of almost normal subgroups are those which are normal, or finite, or have a finite index.  A sample of non-trivial  Hecke pairs will be given in Examples \ref{exs:Hecke}. 

Let us recall that to be commensurate is an equivalence relation on the set of subgroups of $G$. Moreover, if $H$ and $K$ (resp. $H_1$ and $K_1$) are commensurate, then $H\cap H_1$ and $K \cap K_1$ are commensurate. It follows that almost normality is stable by finite intersection and under the above equivalence relation. Moreover, the properties of being finite or of finite index are preserved under this equivalence relation.

  F.s.o.  $G$-actions are closely related to homomorphisms from $G$ into totally disconnected lo\-cally compact ({\it t.d.l.c.}) groups, with dense image. Let us recall the following result.

\begin{thm}\label{thm:exist_proper}{$($\cite[Theorem 1.3]{AD13}$)$} Let $G$ be a  group acting on a  set $X$.  Let $\rho$ be the corresponding homomorphism from $G$ into $\Sym(X)$ and denote by
$G'$ the closure of $\rho(G)$ in ${\rm Map}(X)$.   Then $G'$ is a subgroup of $\Sym(X)$ acting properly on the discrete space $X$ if and only if the $G$-action is f.s.o. Moreover, in this case the group $G'$ is locally compact   and totally disconnected.
\end{thm}

 We will say that $G'$ is the {\it Schlichting group} (or the {\it Schlichting completion}) {\it associated with the action}, since this completion of $\rho(G)$ was first studied by Schlichting \cite{Sch79,Sch80} for left actions $G\actson G/H$ where $H$ is an almost normal subgroup of $G$.

 \subsection{Embeddings into   t.d.l.c. groups} 
We first gather in the next lemma, for further use, the following facts, whose straightforward proof is omitted.

\begin{lem}\label{lem:straight} Let $\mathfrak G$ be a locally compact group acting properly and faithfully on a set $X$ and let $G$ be a dense subgroup of $\mathfrak G$. 
\begin{itemize}
\item[(1)] For $x\in X$, the stabilizer  $G_x = \mathfrak G_{x}\cap G$ relative to the $G$-action is dense in $\mathfrak G_{x}$.
\item[(2)] We  have $\mathfrak Gx = Gx$ for every $x \in X$, and therefore for every $g'\in \mathfrak G$ there is a unique $gG_x$ in $G/G_x$ such that $g' = gh'$ with $h'\in \mathfrak G_{x}$. In particular $g'\mathfrak G_{x} g'^{-1} = g\mathfrak G_{x}g^{-1}$.
\item[(3)] The canonical map $gG_x \mapsto g \mathfrak G_{x}$ from $G/G_x$ into $\mathfrak G/\mathfrak G_{x}$ is a $G$-equivariant bijection. 
\end{itemize}
\end{lem}

 Let $\rho : G \to \Sym(X)$ be the homomorphism associated to an action $G\actson X$. This action  is {\it faithful} ({\it i.e.}, $\rho$ is injective) if and only if $\cap_{x\in X} G_x = \set{e}$.  If moreover it is f.s.o.,  the stabilizers $G_x$ are residually finite groups since $(G_x\cap G_y)_{y\in X}$ is a family of finite index subgroups of $G_x$ with trivial intersection.

 It follows from the theorem \ref{thm:exist_proper} that every faithful and f.s.o. action $G\actson X$ gives rise to an embedding of $G$ into  the t.d.l.c. group $G'$. The cases where $G'$ is compact and where $G' = G$ are described below.
 
 \begin{exs}\label{exs:scases}  Let $\rho: G \to \Sym(X)$ be a faithful   f.s.o. action. We identify $G$ with its image.
 
 (1)  $G'$ is compact if and only if one of the orbits (and then every orbit) of $G\actson X$ is finite. Indeed, assume that the orbit of some $x_0$ is finite. By Lemma \ref{lem:straight} (3) the set $G'/G_{x_0}'$ is finite and therefore $G'$ is compact. The converse is obvious since $G'$ has finite orbits when it is compact. 
  
 (2)  $G' = G$ if and only if one stabilizer  of the action is finite. Indeed assume that $G_{x_0}$ is finite. This implies that $G_{x_0} = G'_{x_0}$ is a finite open neighborhood of $\set{e}$ in $G'$. Therefore $G'$ is discrete and so $G = G'$. The converse is obvious since $G'$ acts properly.
 \end{exs} 
 
 We describe now all the ways to  embed $G$ into a t.d.l.c. group. We will use the following definition.
 
\begin{defn}\label{Gsimple} Let $H, K$ be two subgroups of $G$. The {\it $K$-core} of $H$  is its subgroup $\cap_{k\in K} kHk^{-1}$. The $G$-core of $H$ is usually called the {\it core} of $H$.
\end{defn}

 Note that the core of $H$ is trivial if and only if $H$ does not contain any normal subgroup of $G$ except the trivial one, or equivalently if and only if the action $G\actson G/H$ is faithful.
 
 \begin{rem}\label{rem:commensurate} The above definition, as well as Definition \ref{defn:com},  makes sense for any pair of closed subgroups of a locally compact group. In particular, any two open compact subgroups are commensurate. 
 \end{rem}
 
Let $\rho: G \to \Sym(X)$ be a  f.s.o. action.  
  If $I$ is a subset of $X$ which meets each orbit in one and only one point we may write $X$ as $\sqcup_{i\in I} G/H_i$ where $H_i$ is the stabilizer of $i$. Then  $\mathbold{H} = (H_i)_{i\in I}$ is a family of mutually commensurate almost normal subgroups  of $G$. The action is faithful   if and only if $\bigcap_{g\in G,i\in I}gH_ig^{-1} = \set{e}$, or equivalently if and only if the group $\cap_{i\in I} H_i$ has a trivial core. In this case, we will also say that {\it the family  $(H_i)_{i\in I}$ has a trivial core}. 

 Conversely, given a  family $\mathbold H = (H_i)_{i\in I}$ of  mutually commensurate almost normal subgroups of  $G$ with trivial core, the $G$-action on $X=\sqcup_{i\in I} G/H_i$ is faithful and f.s.o. The corresponding Schlichting group is denoted by ${G(\mathbold{H})}$ and the embedding of $G$ into ${G(\mathbold{H})}$ is denoted by $\rho_{\!\mathbold H}$.\footnote{If the family is reduced to one subgroup $H$, we will use the notation $G(H)$ and $\rho_H$.} Using Theorem \ref{thm:exist_proper},   Lemmas \ref{lem:top} and \ref{lem:straight}, we observe that if $H'_{i}$ is the closure of $H_i$ in $G(\mathbold{H})$, then $(H'_{i})_{i\in I}$ is a family of compact open subgroups of $G(\mathbold{H})$ such that $\bigcap_{g\in G,i\in I}g H_{i}'g^{-1} =  \bigcap_{g'\in G',i\in I}g' H_{i}'g'^{-1} = \set{e}$. In other words, the family $(H_{i}')_{i\in I}$  has a trivial core. We also observe that the family of finite intersections of open compact sets of the form $g H_{i}'g^{-1}$ is a basis of neighborhoods of $e$ in ${G(\mathbold{H})}$. Finally, we observe that ${G(\mathbold{H})}$ is second countable whenever $I$ and $G$ are countable. Indeed, in this case, ${G(\mathbold{H})}$ is metrizable and separable.

 \begin{prop}\label{prop:subject} Let $\rho$ be an injective homomorphism from a group $G$ into a t.d.l.c. group $\mathfrak G$, with dense range\footnote{$\mathfrak G$ is called a {\it t.d.l.c. completion} of $G$.}.  There exist a  family $\mathbold H = (H_i)_{i\in I}$ of  mutually commensurate almost normal subgroups of  $G$ with trivial core and an isomorphism $\theta : \mathfrak G \to  G(\mathbold{H})$ such that $\theta \circ \rho = \rho_{\mathbold{H}}$. Moreover, when $\mathfrak G$ is first countable, we may choose $I$ to be countable.
 \end{prop}

 \begin{proof} 
 By the van Dantzig's theorem (see {\it e.g.} \cite[Theorem II.7.7]{HR}) we know that the set of all compact open subgroups of $\mathfrak G$ is a basis of neighborhoods of the identity. We choose a  family $(H_{i}')$ of compact open subgroups (with $I$ countable when $\mathfrak G$ is first countable) such that $\bigcap_{g'\in \mathfrak G}g'(\cap_{i\in I} H_{i}')g'^{-1} = \set{e}$ and we set $X = \sqcup_{i\in I} \mathfrak G/H_{i}'$.  Then $\mathfrak G$ acts properly and faithfully on $X$. In particular, the topology on $\mathfrak G$ is the topology of pointwise convergence. Observe that the groups $g'H_{i}'g'^{-1}$ are mutually commensurate. The $\mathfrak G$-action on $X$ is f.s.o. and therefore the $G$-action is also f.s.o. Then the closure $G'$ of $G$ in ${\rm Map}(X)$ is a t.d.l.c. group which contains $\mathfrak G$, where $\mathfrak G$ is identified with its range  in $\Sym(X)$. But $\mathfrak G$ is a locally compact subgroup of $G'$, hence closed and so $\mathfrak G = G'$. The groups $G$ and $G' = \mathfrak G$ have the same orbits. Therefore $X =  \sqcup_{i\in I} G/H_i$ where $H_i = H_{i}' \cap G$ and $\bigcap_{g\in G}g(\cap_{i\in I} H_i)g^{-1} = \set{e}$. \end{proof}

\begin{ex}\label{ex}  Let $\mathbold{H}$ be a  family of  finite index subgroups of $G$ with trivial core. Then $(\rho_{\mathbold H}, G(\mathbold H))$ is a  compact, totally disconnected completion of $G$, as already said in Example \ref{exs:scases} (1).  Moreover, every such completion of $G$ is obtained in this way.

If $G$ is residually finite and if $\mathbold{H} =(H_i)_{i\in I}$ is a directed family of finite index normal subgroups such that $\cap_{i\in I} H_i = \set{e}$ then $G'= G(\mathbold{H})$ is isomorphic to the profinite completion of $G$ with respect to $(H_i)_{i\in I}$.

 For instance for $G = \Z$,  if we take the family of all proper subgroups of $\Z$ then $G'$ is the group  $\widehat{\Z}$ of $\Z$ of profinite integers. If we take the family $(p^k \Z)_{k\geq 1}$ for $p$ prime, we get the group $\Z_p$ of $p$-adic integers, which is a quotient of $\widehat{\Z} \simeq \prod_p \Z_p$.
\end{ex}

\begin{exs}\label{exs:Hecke} Many interesting  Hecke pairs $(G,H)$ and their completions are found in the literature. Let us give a few examples. Below, $p$ is a prime number and $H'$ is the closure of $H$ in $G'$.
  
$\bullet$ $H = SL_n(\Z)$, $G = SL_n(\Z[1/p])$. Then $G' = SL_n(\Q_p)$ and  $H'=SL_n(\Z_p)$ (see \cite{SW13});

$\bullet$ $H = SL_n(\Z)$, $G= SL_n(\Q)$. Then $G' = SL_n(\cA_f)$ and $H' = SL_n(\widehat{\Z})$, where $\cA_f$ is the ring of finite ad\`eles (see \cite{An});

$\bullet$ $H = \Z\rtimes 1$, $G = \Q\rtimes \Q^*_{+}$. Then  $G' = \cA_f \rtimes  \Q^*_{+}$ and $H' = \widehat{\Z}\rtimes 1$ (see \cite{BC});

$\bullet$ $H = \scal{a}$, $G = BS(m,n) = \scal{a,t| t^{-1} a^m t = a^n}$ the Baumslag-Solitar group with parameters $m,n$. This group $G$ acts faithfully and transitively on its Bass-Serre tree $T$. Its Schlichting completion is the closure $G'$ of $G$ into the group $\Aut(T)$ of automorphisms of $T$ equipped with the topology of pointwise convergence.
\end{exs}

We observe in passing that there exist groups having only finite proper subgroups, called quasi-finite groups (see \cite{Ols}). Such groups cannot be embedded into t.d.l.c. groups, except themselves.

For some other groups, all almost normal subgroups are either finite of have a finite index. This is for instance the case of $SL_n(\Z)$ with $n\geq 3$ (see \cite{SW13}). For such groups , the only possible faithful t.d.l.c. completions, except themselves, are compact.

 \subsection{Transitive t.d.l.c. completions} A  t.d.l.c. completion $(\rho,\mathfrak G)$ of $G$ will be called {\it transitive} if it is of the form $(\rho_H,G(H))$ with respect to an  almost normal subgroup $H$ of $G$ with trivial core.

\begin{prop}\label{prop:eq_trans} Let $\rho$ be an  injective  homomorphism from $G$ into a locally compact  group $\mathfrak G$, with dense image. The following conditions are equivalent:
\begin{itemize}
\item[(i)] $(\rho, \mathfrak G)$ is a  transitive t.d.l.c. completion of $G$;
\item[(ii)] there exists an open compact subgroup $H'$ of $\mathfrak G$ with trivial $G$-core.
\end{itemize}
\end{prop}

\begin{proof}   Let $H$ be an almost normal subgroup  of $G$ with trivial $G$-core such that $\mathfrak G = G(H)$ is the Schlichting group corresponding to the action $G\actson G/H$. Then the closure $H'$ of $H$ in $\mathfrak G$ satisfies the property stated in (ii). Conversely,   starting from a subgroup $H'$ as in (ii), the action $\mathfrak G\actson \mathfrak G/H'$ is faithful and proper. Setting $H = H' \cap G$  we get $(\rho,\mathfrak G) = (\rho_H,G(H))$.
\end{proof}

\begin{cor}\label{cor:transit} Let $\mathfrak G$ be a locally compact group acting properly and faithfully on a set $X$ with a finite number of orbits\footnote{then we say that the action is {\it almost transitive}.}.
Let $G$ be a dense subgroup of $\mathfrak G$ and denote by $\rho:G\to \mathfrak G$ the inclusion. Then $(\rho,\mathfrak G)$ is a   transitive t.d.l.c. completion of $G$.
\end{cor}

\begin{proof} We assume that $X = \sqcup_{1\leq i\leq n} \mathfrak Gx_i$. We set $H' = \cap_{1\leq i\leq n} \mathfrak G_{x_i}$. Then, $H'$ is a compact open subgroup of $\mathfrak G$. Moreover, since the action is faithful, we have $\cap_{g'\in \mathfrak G, i = 1,\dots,n} g'\mathfrak G_{x_i} g'^{-1} = \set{e}$. But each $g'\mathfrak G_{x_i} g'^{-1}$ is of the form $g\mathfrak G_{x_i} g^{-1}$ by Lemma \ref{lem:straight}. Therefore $H'$ has a trivial $G$-core and we apply the previous proposition in order to conclude.
\end{proof}

\subsection{F.s.o. actions under separability assumptions}
Let $G\actson X$ be an action of a group $G$ on a set $X$.  We say that {\it a metric $d$ on $X$ is $G$-invariant, or that the action preserves the metric,} if $d(gx,gy) = d(x,y)$ for every $g\in G$ and $x,y\in X$. One also says that $G$ {\it acts by isometries on} $(X,d)$.

 Observe that a locally finite metric space is countable. The following converse of the  theorem \ref{thm:proper} is a particular case  of \cite[Theorem 4.2]{A-M-N}.

\begin{thm}\label{AMN} Let $\mathfrak G$ be a locally compact group acting properly on a countable set $X$. Then, there is a $\mathfrak G$-invariant locally finite metric on $X$.
\end{thm}

   Any action $G\actson X$ which preserves a locally finite metric  is f.s.o.  The  converse holds when $X$ is countable, as an immediate consequence of the previous theorem and of Theorem \ref{thm:exist_proper}.

\begin{prop}\label{prop:exist_proper} Let $G$ be a  group acting on a countable set $X$.    The following conditions are equivalent:
\begin{itemize}
\item[(i)] there exists a $G$-invariant locally finite metric $d$ on $X$;
\item[(ii)] the action is f.s.o.
\end{itemize}
\end{prop}

 \begin{rem}\label{rem:transit}Let $H$ be a subgroup of $G$ such that $G/H$ is countable. Observe that there exists a locally finite $G$-invariant metric on $G/H$ if and only if $H$ is an almost normal subgroup. For a locally compact group $\mathfrak G$, the existence of $\mathfrak G$-invariant (for the left $\mathfrak G$-action), proper  ({\it i.e.,} with compact balls), compatible ({\it i.e.,} defining the topology of $\mathfrak G$) metrics on $\mathfrak G$ has been studied in the seventies. Struble \cite{Struble} proved that a locally compact group $\mathfrak G$ has a proper  $\mathfrak G$-invariant  compatible metric if and only if it is second countable. This result was rediscovered in \cite{H-P}.  Of course, when a group $G$ is finitely generated, one may take the geodesic distance on the Cayley graph associated with any finite system of generators. 
\end{rem}

Let $G\actson X$ be an action that preserves a locally finite metric. In general, there does not exist a structure of locally finite connected graph having $X$ as set of vertices and such that the $G$-action is given by automorphisms of the graph. For instance let us consider a countable, non-finitely generated group $G$, acting on itself by translations. There exists on $G$ a $G$-invariant locally finite metric (for a concrete construction see \cite[\S 2]{Tu}), but this $G$-action does not preserve a  locally finite connected graph structure. On the other hand, when  $G$ is finitely generated we have the following result.

\begin{prop}\label{prop:graph} Let  $G$ be a finitely generated group and let  $G\actson X$ be a f.s.o. action on a countable set $X$. Then there exists a structure of locally finite connected graph having $X$ as set of vertices and such that the $G$-action is given by automorphisms of the graph.
\end{prop}

\begin{proof} We first assume that the action is transitive: there exists an almost normal subgroup  $H$ of $G$ such that $G$ acts by left translations on $X = G/H$. Let $S = S^{-1}$ be a finite symmetric set of generators of $G$ not containing the unit $e$ of $G$. Since $H$ is almost normal, there is a finite subset $\set{g_1,\dots, g_l}$ of $G$ such that
$$\bigcup_{s\in S} HsH = \bigsqcup_{1\leq i \leq l} g_iH.$$
Let $\Gamma = (G/H, E)$ be the graph such that a pair $\set{fH,gH}$, with $fH\not= gH$, is an edge if and only if $g^{-1} f \in \bigcup_{s\in S} HsH$. This graph is connected, locally finite (each vertex having the same valency $\leq l$) and $G$ acts by automorphisms on it.

Let us consider now the general case. We have $X = \sqcup_{n\geq 1} G/H_n$, where the $H_n$ are mutually commensurable almost normal subgroups of $G$. We equip each $G$-invariant subset $G/H_n$  with a structure of locally finite connected graph as in the first part of the proof. For every $g\in G$, we put an edge between $gH_1\in G/H_1$ and $gH_2\in G/H_2$. The number of edges connecting $gH_1$ to some element of $G/H_2$ is finite. Indeed, the action of $H_1$ on $G/H_2$ has finite orbits, so $H_1H_2 = \sqcup_{1\leq i \leq l} g_i H_2$. Now we observe that there is an edge connecting $H_1$ and $gH_2$ if and only if there is some $h\in H_1$ such that $hH_2 = gH_2$. It follows that $gH_2 = g_iH_2 \in G/H_2$ for some $i \in \set{1,\dots,l}$. Similarly, one sees that the number of edges connecting any $gH_2\in G/H_2$ to some element of $G/H_1$ is finite. Finally, for every $n$, we put in the same way edges connecting the elements of $G/H_n$ to suitable elements of $G/H_{n+1}$. This defines a structure of locally finite connected graph having $X$ as set of vertices and such that the $G$-action is given by automorphisms of the graph.
\end{proof}

 \begin{rem}  We have seen in Proposition \ref{prop:subject} that all the embeddings of a group $G$ as dense subgroup of a t.d.l.c. group are provided by  families of mutually commensurate almost normal subgroups of $G$ with trivial core. It may be useful to understand when two such families define the same embedding (up to isomorphism in the obvious sense). When $G$ is countable, two countable such families $\mathbold{H} = (H_i)_{i\in I}$ and $\mathbold{K} = (K_j)_{j\in J}$ give the same embedding if and only if 
 \begin{itemize}
 \item the $H_i$ and $K_j$ are  commensurate for every $i\in I$ and $j\in J$;
 \item every $H_i$ contains a finite intersection of conjugates of the $K_j$  and  every $K_j$ contains a finite intersection of conjugates of the $H_i$.
 \end{itemize}
This result will not be needed in our paper and we omit the proof.
 \end{rem}

\section{Amenable f.s.o. actions}\label{sect:amenact}

\subsection{Some characterizations and consequences} F.s.o. amenable actions are characterized by the following properties.

 \begin{prop}\label{co_sco3}  Let $G\actson X$ be a f.s.o. action of a group $G$ on $X$.  Let $G'$ be the corresponding Schlichting group. The following conditions  are equivalent:
 \begin{itemize}
 \item[(i)] the global action $G\actson X$ is amenable;
 \item[(ii)] there exists a $G$-orbit $Y$ such that the action $G\actson Y$ is amenable;
 \item[(iii)] the  action of $G$ on each of its orbits is amenable;
 \item[(iv)] $G'$ is amenable.
  \end{itemize}
 \end{prop}
 
 \begin{proof} (iv) $\Rightarrow$ (iii) $\Rightarrow$ (ii) $\Rightarrow$ (i) is obvious. Let us show that (i) $\Rightarrow$ (iv). Let $m$ be a $G$-invariant mean on $\ell^\infty(X)$. We denote here by $\ell^\infty(X)^c$ the subalgebra of $\ell^\infty(X)$ formed by the functions $f$ such that $g\mapsto _{g}\!\!f$ is continuous from $G'$ into $\ell^\infty(X)$. Since $G$ is dense in $G'$, we see that $m$ is a $G'$-invariant mean on $\ell^\infty(X)^c$. It follows that the action of $G'$ on $X$ is amenable. Since the stabilizers of this action are compact, hence amenable, it follows from Lemma \ref{lem:useful} that $G'$ is amenable.
 \end{proof}

\begin{rem} Let $G\actson X$ be an amenable f.s.o. action of a non-amenable group. Then the stabilizers $G_x$ are very large in the sense that for every $g\in G$ the group $gG_xg^{-1} \cap G_x$ is not amenable. Indeed otherwise $G_x$ would be amenable since $G_x$ is almost normal, and $G$ would be amenable since the action $G\actson G_x$ is amenable.
\end{rem}

\begin{cor}\label{cor:Eymard} Let $G\actson X$ be an amenable f.s.o. action and let $H$ be a subgroup of $G$.  Then the action of $H$ on its $H$-orbits is amenable. 
\end{cor}

\begin{cor} Let $K$ be a co-amenable almost normal subgroup of $G$ and let $H$ be a subgroup of $G$. Then $H\cap K$ is a co-amenable almost normal subgroup of $H$.
\end{cor}

\begin{proof} We apply the previous corollary to the action of $G$ on $G/K$ and $x_0 = K\in G/K$. So the action of $H$ on the orbit $Hx_0$ is amenable. The stabilizer of $x_0$ in $H$ is $H\cap K$ and we see that
$H\cap K$ is a co-amenable almost normal subgroup of $H$.
\end{proof}

\begin{cor} Let $\widetilde{\cA}_{1}$ be the family of countable groups that admit a faithful amenable f.s.o. action. A group $G$ belongs to $\widetilde{\cA}_{1}$ if and only if it is a subgroup of a t.d.l.c. amenable group.
\end{cor}

\begin{proof} Use Propositions \ref{prop:subject} and \ref{co_sco3}. \end{proof}

\begin{rem}\label{rem:sep} Let $G\in \widetilde{\cA}_{1}$.  Since $G$ is countable,  it acts faithfully and amenably on a locally finite metric space. Indeed, let $G\actson X$ be a faithful, amenable f.s.o. action. Then there is a $G$-invariant countable subset $Y$ on which the restricted action still has the same properties. It follows that $G$ embeds in a second countable t.d.l.c. amenable group.
\end{rem}

\begin{rem}\label{rem:obstruct}  The class $\widetilde{\cA}_{1}$ includes the amenable countable groups and the residually finite groups. However there are also countable groups that do not belong to $\widetilde{\cA}_{1}$.  It is the case for exotic groups like non-amenable quasi-finite groups. 

A less exotic obstruction is provided by the property (T). 
Indeed, assume that $G$ has Property (T) and let    $G\actson X$ be a faithful, amenable, f.s.o. action. Since the $G$-action on every orbit $Gx$ is amenable, the trivial representation of $G$ is weakly contained into the representation of $G$ in $\ell^2(Gx)$ by left translations. If follows that there is a non-zero $G$-invariant vector in $\ell^2(Gx)$ since $G$ has Property (T). Therefore $Gx$ is finite for every $x\in X$ and, as a result, $G$ is residually finite. Similarly, if $G \in \widetilde{\cA}_{1}$, every Property (T) subgroup must be residually finite.
 \end{rem}
 
\subsection{Amenable actions on locally finite graphs} The proposition \ref{co_sco3} applies in particular to group actions by automorphisms of a locally finite connected graph $\Gamma = (X,E)$. A subgroup $G$   of $\Aut(\Gamma)$ acts amenably on the set $X$ of vertices of $\Gamma$ if and only if the closure $G'$ of $G$ in $\Aut(\Gamma)$ is an amenable locally compact group. Necessary or sufficient conditions  for $G\actson X$ to be amenable have been studied by several authors.  In order to state their results, we need to introduce the notion  of end of a graph $\Gamma$. A {\it ray} (or {\it half-line}) is a sequence $(x_0, x_1, \dots)$ of successively adjacent vertices without repetitions. Two rays $R_1$ and $R_2$ are said to be in the same {\it end} if there is a ray $R_3$ that contains infinitely many vertices in $R_1$ and in $R_2$.  We denote by $[x_0, x_1, \dots]$ the end corresponding to $(x_0, x_1, \dots)$. If an end contains infinitely many disjoint rays, it is called {\it thick}. Otherwise, one says that the end is {\it thin}. In particular, when $\Gamma$ is a tree, two rays are in the same end if and only if their intersection is a ray. In this case every end is thin. 

\begin{thm}\label{Woess} Let $\Gamma = (X,E)$ be a   locally finite connected graph, and let $G$ be any subgroup of $\Aut(\Gamma)$. 
\begin{itemize} 
\item[(i)] If the action of $G$ on $X$ is amenable then one of the following three statements holds:
\begin{itemize}
\item[(a)] $G$ fixes a finite subset of $X$;
\item[(b)] $G$ fixes an end of $X$;
\item[(c)] $G$ fixes a set of two ends which are the directions of a hyperbolic automorphism and its inverse.
\end{itemize}

\item[(ii)] Moreover if $G\actson X$ is  amenable, then
\begin{itemize}
\item[(d)] $G$  does not contain any  (non-abelian) free group, {\rm discrete} in $\Aut(\Gamma)$;
\item[(e)] $G$  does not contain any (non-abelian) free group {\rm acting freely}  on $X$.
\end{itemize}
\item[(iii)] Conversely, if $($a$)$ or $($c$)$ holds or if $G$ fixes a thin end, then the action of $G$ is amenable.
\item[(iv)] When $\Gamma$ is a tree, then each of the conditions $($d$)$ and $($e$)$ is equivalent to the amenability of the action.
\end{itemize}
\end{thm}

Recall that an automorphism $\gamma$ is hyperbolic if it leaves no finite subset of vertices in $X$ invariant and  fixes a unique pair of ends $\omega_0$, $\omega_1$. These ends are the directions of $\gamma$ and of $\gamma^{-1}$: one has $\lim_{n} \gamma^n x = \omega_0$ and $\lim_{n} \gamma^{-n} x = \omega_1$ for $x\in X$, in the topology of the compactification of $X$ by the ends of $\Gamma$. 

The items (i) and (iii) are due to Woess \cite[Theorems 1 and 2]{Woess89}.  Observe that $G'$ is a closed amenable subgroup of $\Aut(\Gamma)$ when $G\actson X$ is amenable, and therefore $G$ cannot contain a discrete free subgroup of $\Aut(\Gamma)$. Moreover (d) implies (e) since any subgroup acting freely  on $X$ is discrete. The assertion (iv) is due to Nebbia  \cite{Nebbia} and Pays-Valette \cite{PV}. It may happen that $G$ fixes a thick end whereas $G$ does not act amenably, as shown by an example in \cite{Woess89}. 

\begin{rem}\label{rem:free} It is interesting to note that if a closed group $\mathfrak G$ of automorphisms  of a   locally finite   connected graph $\Gamma = (X,E)$ satisfies none of (a), (b), (c), then  it contains the free group $\F_2$ on two generators as a discrete subgroup (see   \cite{Nebbia},  \cite[Theorem 8]{Mo92}, \cite{Jung}). It is in particular the case if $\Gamma$ only has thin ends (for instance is a tree) and $\mathfrak G$ is not  amenable, by (iii) in the previous theorem and Proposition \ref{co_sco3}.
\end{rem}

\section{Faithful, transitive,  amenable f.s.o. actions}\label{sect:transitive}
We denote by ${\mathcal A}_{1}$ {\it the class of countable groups that admit a faithful, transitive, amenable f.s.o. action}, or equivalently {\it that admit a co-amenable almost normal subgroup with trivial core}. 

\subsection{Some properties of the class ${\mathcal A}_{1}$} We first give a characterization of this class.
\begin{prop}\label{prop:equiv_alnor} Let $G$ be a countable group. The following conditions are equivalent:
\begin{itemize} 
\item[(i)] $G\in {\mathcal A}_{1}$;
\item[(ii)] $G$ embeds densely in a  locally compact amenable group $\mathfrak G$ that contains an open compact subgroup  with trivial $G$-core;
\item[(iii)] $G$ embeds densely in a  locally compact amenable group $\mathfrak G$ that has a proper, faithful and almost transitive action on a  set $X$.
\end{itemize} 
\end{prop}

\begin{proof} (i)  $\Rightarrow$ (iii) by Theorem \ref{thm:exist_proper} and Proposition \ref{co_sco3};
(iii) $\Rightarrow$ (ii) by Corollary \ref{cor:transit} and Proposition \ref{prop:eq_trans} and
(ii) $\Rightarrow$ (i) by Propositions \ref{prop:eq_trans} and \ref{co_sco3}.
\end{proof}

 Obviously, this class ${\mathcal A}_{1}$ contains the class of countable amenable groups. We will see later that it also contains non-amenable groups. Before, we study some properties of this class of groups.

It is easily checked that a product $G\times H$ is in $\cA_1$ if and only if $G$ and $H$ are in $\cA_1$ (same proof as in \cite[Section 4.A]{GM}).  As we will see in Section \ref{sect:const_ex},  a group in the class $\cA_{1}$ may contain infinite groups with the Kazhdan property (T). So to be in ${\mathcal A}_{1}$ is not hereditary. However, we have the following fact.

\begin{prop}\label{prop:fin_ind} Let $G \in {\mathcal A}_{1}$. Then every subgroup $K$ of finite index in $G$ is still in ${\mathcal A}_{1}$.
\end{prop}

\begin{proof} Let $G\actson X$ be a faithful, transitive, amenable f.s.o. action. The action of $K$ on $X$ has finitely many orbits. The closure $\cK'$ of $K$ in $\Map(X)$ is a subgroup of the closure $G'$ of $G$. Since $G'$ is amenable and acts properly on $X$, its closed subgroup $\cK'$ has the same properties. Then we apply  Proposition \ref{prop:equiv_alnor}.
\end{proof}

Note that to the author's knowledge, it is still unknown whether the class $\mathcal A$ is closed under passing to finite  index subgroups.

 We now study some other properties  of subgroups of groups in the class ${\cA}_1$. Let $\cC$ be a class of groups. We say that a group $G$ is {\it residually} $\cC$ if for every $g\not= e$ in $G$, there exists a normal subgroup $N$ of $G$ such that $g\notin N$ and $G/N$ belongs to the class $\cC$. When $\cC$ is the class of finite groups, we say that $G$ is {\it residually finite}. When $\cC$ is the class of finite $p$-groups, we say that $G$ is  {\it residually a finite $p$-group}.

We denote by $\cA_{fin}$ the class of finite groups.

\begin{prop}\label{prop:resid}  Let $G$ be a group in the class $\cA_{1}$ and let $G_1$ be a subgroup of $G$. Then $G_1$ is residually ${\mathcal A}_1\cup \cA_{fin}$. 
\end{prop}

\begin{proof} Let $G\actson X$ be a faithful, transitive, amenable f.s.o. action. We write $X$ as a union of $G_1$-orbits: $X=\sqcup_i Y_i$.  Denote by $\rho_i$ the canonical surjective homomorphism from $G_1$ onto $(G_1)_{|Y_i}$ and by $N_i$ the kernel of $\rho_i$. Since the action of $G_1$ on $Y_i$ is amenable, each  group $G_1/N_i$ is in ${\mathcal A}_1\cup \cA_{fin}$. Moreover, if $g\in G_1$ is a non-trivial element, there exists some index $i$ such that $g\notin N_i$. \end{proof}

The class $\cA$ is strictly larger than the class $\cA_{1}$. For instance, Glasner and Monod \cite{GM} have shown that for any countable group $G$, the free product $G*H$ is in $\cA$ as soon as $H$ is countable and residually finite ({\it e.g.} $\F_2$). But for instance, if $G$   is non-residually finite with the Kazhdan property, then $G*H$ is not in $\cA_{1}$.

\subsection{Groups acting faithfully, transitively and amenably on graphs}  Recall that this fami\-ly of groups includes all finitely generated groups that belong to $\cA_{1}$. 

 Let $\Gamma$ be a  locally finite  connected graph. We say that $\Gamma$ is {\it transitive} (resp. {\it almost transitive}) if $\Aut(\Gamma)$ acts transitively (resp. with a finite number of orbits) on its set of vertices. Such graphs have either one, two or infinitely many ends (\cite{Ab74}, \cite[Corollary 1]{Mo}, \cite[Theorem 6]{Kro}). 

Another important feature of an almost transitive graph $\Gamma$ is its growth, which is related to its number of ends. Let $x_0\in X$ and, for every integer $n\geq 1$, denote by $a_n$ the number of vertices of $\Gamma$ at distance at most $n$ from $x_0$. The {\it growth} of $\Gamma$ is the rate of growth (which does not depend on the choice of $x_0$) of the sequence $(a_n)$. Then (see \cite[Theorem 9.3]{Cam} for references, and \cite{IS91}):
{\it 
\begin{itemize} 
\item if $\Gamma$ has infinitely many ends its has exponential growth ({\it i.e.} $\lim_n a_n^{1/n} > 1$);
\item $\Gamma$ has linear growth if and only if it has two ends (and these ends are thin);
\item a graph with one end displays all possible rates of growth (polynomial, exponential, intermediate), except linear growth (intermediate meaning that $(a_n)$ grows faster than any polynomial but slower than any sequence $(\lambda^n)$ with $\lambda >1$).
\end{itemize}}

 Let us also mention the following result  contained in \cite{SW}. If $F$ is a finite subset of $X$, we define its boundary $\partial F$ to be the set of all vertices in $X\setminus F$ that have a neighbor in $F$. The {\it isoperimetric number} $i(\Gamma)$ of $\Gamma$ is the infimum of $\abs{\partial F}/\abs{F}$, where $F$ runs over the set of non-empty finite subsets of $X$. Assume that $\Gamma$ is transitive. Then the following properties are equivalent:
\begin{itemize}
\item $i(\Gamma) = 0$;
\item some (equivalently every) closed transitive subgroup of $\Aut(\Gamma)$ is amenable and unimodular.
\end{itemize}
This applies in particular if $\Gamma$ has  subexponential growth ({\it i.e.} $\lim_n a_n^{1/n} = 1$) since we have $i(\Gamma) = 0$ in this case. Observe that one may have $i(\Gamma) = 0$ whereas $\Gamma$ has exponential growth (consider for instance the Cayley graph of a solvable finitely generated group with exponential growth).

 Every countable group acting faithfully and transitively on a graph having subexponential growth (in particular a graph with two ends) is in $\cA_1$. One-ended transitive graphs  are the most difficult to handle. Let us now focus on transitive graphs with infinitely many ends. In this case, subgroups of $\Aut(\Gamma)$ acting transitively and amenably have a nice characterization.
 
 \begin{thm}\label{thm:SW} Let $\Gamma = (X,E)$ be a    locally finite connected graph and let $G$ be a subgroup of $\Aut(\Gamma)$ which acts transitively on $X$. We assume that $\Gamma$ has infinitely many ends.  Let $G'$ be the closure of $G$ in $\Aut(\Gamma)$. The following conditions are equivalent:
 \begin{itemize}
 \item[(i)] $G'$ is amenable;
\item[(ii)] $G'$ (or equivalently $G$) fixes an end (and this fixed end is unique and thin);
 \item[(iii)] the  action $G\actson X$ is amenable.
 \end{itemize}
\end{thm}

\begin{proof} The equivalence between (i) and (ii) is \cite[Proposition 2]{SW}. The equivalence between (i) and (iii) was proved in Proposition \ref{co_sco3}. 
\end{proof}

Let us observe that when the conditions of this theorem are satisfied, the graph  is not very different from a tree.

\begin{thm}\label{thm:qitree} {$($\cite[Theorem 1]{Mo92II}$)$} Let $\Gamma = (X,E)$ be a  locally finite connected   graph with infinitely many ends. Assume that there is an end $\omega$ such that the stabilizer $\Aut(\Gamma)_\omega$ of this end acts transitively on $X$. Then $\Gamma$ is quasi-isometric to a tree.
\end{thm}

\begin{rem}\label{rem:non amen} Let $G$ be an infinite and finitely generated group. The number of ends of $G$ is defined as the number of ends of its Cayley graph with respect to some finite generating set of $G$, and therefore $G$ has either one, two or infinitely many ends. It has two ends if and only if it has an infinite cyclic group of finite index. The case of infinitely many ends is described by the theorem of Stallings (see \cite{Sta68}). From this theorem one deduces that a  group $G$ with infinitely many ends contains a non-abelian free subgroup and therefore is not amenable. 

Faithful, transitive, amenable actions of a group $G$ by automorphisms of a locally finite connected graph with infinitely many ends were considered in \cite{MV}. It is proved in this paper that such actions of $G$ (as a discrete group) are not proper. This gives a direct proof of the non-amenability of a group having infinitely many ends. Another proof of this latter fact is also given in \cite{SW}.
\end{rem}

The notion of Cayley graph can be generalized to compactly generated locally compact groups.

\begin{defn} Let $\mathfrak G$ be a locally compact group. A  locally finite connected graph $\Gamma$ is called a {\it Cayley-Abels} (or {\it rough Cayley}) graph for $\mathfrak G$ if $\mathfrak G$ acts transitively on $\Gamma$ and if the stabilizers of vertices are compact open subgroups of $\mathfrak G$ (see \cite[Definition 2.E.10]{CH}, \cite[Definition 2.1]{KM}).
\end{defn}

A t.d.l.c. group has a Cayley-Abels graph if and only if it is compactly generated \cite[Theorem 2.2]{KM}. Moreover, in this case any two Cayley-graphs are quasi-isometric \cite[Theorem 2.7]{KM}. Then the ends  of $\mathfrak G$ are  well defined as the ends of any of its Cayley-Abels graphs.

Theorem \ref{thm:SW} can also be proved by using Abels' analysis of the structure of compactly generated locally compact groups having infinitely many ends. Indeed, with the notation of this theorem, $\Gamma$ is a Cayley-Abels graph of $G'$. Then, by \cite[Theorem 4]{Ab77}, the group $G'$ is amenable if and only if it fixes one end of $\Gamma$, and in this case $G'$ is a $HNN$-extension $HNN_\alpha(H)$ (see the definition in Section \ref{charac}), where $H$ is a compact group and $\alpha : H\to H$ is an injective, non-surjective, open continuous homomorphism. Moreover, this case occurs. A simple example is  given by $H=\Z_p$, $p$ prime,  $\alpha$ being the multiplication by $p$ (see also \cite[Section 3.9]{Ab77}). Compare with the situation in Remark \ref{rem:non amen}.

Note that the case of a t.d.l.c. compactly generated group $\mathfrak G$ with two ends is also well understood \cite[Theorem 4]{Ab77}: it has a compact open normal subgroup $N$ such that $\mathfrak G/N$ is either isomorphic to $\Z$ or to the infinite dihedral group $\Z_2 *\Z_2$. It is therefore amenable and unimodular, as  seen before in this text. 

As already said, we are mainly interested in the existence of non-amenable groups $G$ in this class ${\mathcal A}_{1}$.  
We concentrate on the case of group actions on trees, which is the easiest to consider.

\subsection{Characterization of the elements of ${\mathcal A}_{T_q}$}\label{charac} For every integer $q\geq 2$, we denote by $T_q$ the regular tree where each vertex has degree $q+1$, and  by $\cA_{T_q}$ the class of countable groups having a faithful, transitive and amenable action on the tree $T_q$. In what follows, we could use Abels' results from \cite{Ab77}, but we prefer to give direct elementary proofs.

\begin{defn}  Let $H$ be a group  and $\alpha$ an injective endomorphism. Then the $HNN$-extension 
$$G = HNN_\alpha(H) = \scal{H,t| t^{-1}ht = \alpha(h), h\in H}$$
is called an {\it ascending $HNN$-extension} (or the {\it mapping torus of} $\alpha$).
\end{defn}

\begin{prop}\label{prop:HNN0} Let  $G = HNN_\alpha(H)$ be an ascending $HNN$-extension as above, such that $[H:\alpha(H)]  = q$ with $q\in \N$, $q\geq 2$. Then $G$ acts transitively and amenably on its Bass-Serre tree $T_{q}$.
This action is faithful if and only if the subgroup $\cap_{n\geq 1} \alpha^n(H)$ has a trivial $H$-core.
\end{prop}

\begin{proof} The Bass-Serre tree of this $HNN$-extension is the regular tree $T_q$. Its set of vertices is $X = G/H$ and its set of edges  is $G/H \sqcup G/K$ where $K = \alpha(H)$. For $g\in G$, the directed edge $gK$ has $gH$ as origin and $gt^{-1}H$ as extremity. The directed edge $gH$ has $gH$ as origin and $gtH$ as extremity. The edge opposite to $gK$ is $gt^{-1}H$. For $n\in \Z$, let us denote by $x_n$ the vertex $t^{n}H$. Then $\set{x_n}_{n\in \Z}$ is a line in $T_q$. We have $t x_n = x_{n+1}$ for $n\in \Z$ and $hx_n = x_n$ for $n\geq 0$ and $h\in H$. It follows that $G$ fixes the end $[x_0, x_1, \dots, x_n,\dots]$. Therefore, $G$ acts amenably on $T_{q}$, by Theorem \ref{thm:SW}.

The last assertion follows from the fact that the intersection of the vertex stabilizers is the $H$-core of $\cap_{n\geq 1} \alpha^n(H)$.
\end{proof}

\begin{prop}\label{HNN1} Every group $G$ in $\cA_{T_q}$, $q\geq 2$, is an ascending $HNN$-extension of the form described in the previous proposition, where the subgroup $\cap_{n\geq 1} \alpha^n(H)$ has a trivial $H$-core.
\end{prop}

\begin{proof} The group $G$ acts faithfully and transitively on  $T_{q}$ and fixes an end that we denote by $ \omega = [y_0, y_1, \dots, y_n,\dots]$. For every $g\in G$ there exists an element $\theta(g)\in \Z$ and an $n_0\in \N$ such that
$gy_n = y_{n+\theta(g)}$ for $n\geq n_0$. Moreover, $\theta$ is a homomorphism from $G$ into $\Z$. Since the action is transitive, there is an element $t\in G$ such that $\theta(t) = 1$. Thus $t$ is a translation of step $1$ along a double infinite path $(x_n)_{n\in \Z}$ with $ \omega = [x_0, x_1, \dots, x_n,\dots]$. For $n\in \Z$ we denote by $G_n$  the stabilizer of $x_n$ and we set $H = G_0$.  Since $G$ fixes $\omega$, the sequence $(G_n)_{n\in \Z}$ of subgroups is increasing. We have of course $G_{n+1} = tG_n t^{-1}$. Let $\alpha$ be the isomorphism $h\mapsto t^{-1}ht$ from $H$ onto its subgroup $G_{-1}$ that we denote by $K$. 

We claim that $[H: K] = q$. Denote by $V$ the set of vertices of $T_q$ at distance $1$ from $x_0$. We have $Hx_{-1} \subset  V$ and $[H: K]$ is the number of elements of $Hx_{-1}$. We first observe that $x_1\notin Hx_{-1}$. Otherwise we would have $hx_{-1} = x_1 = hx_1$  for some $h\in H\subset G_1$. Next, we show that for every $v\in V$, $v\not=x_1$, there exists $h\in H$ such that $hx_{-1} = v$. Since the $G$-action on $T_q$ is transitive there exists $g\in G$ such that $gx_{-1} = v$. Let us observe that $\theta(g) = 0$. Indeed, let $n_0$ be such that $gx_n = x_{n+\theta(g)}$ for $n\geq n_0$. We choose $n$ large enough so that $x_n$ and $x_{n+\theta(g)}$ are both on the half-line $(x_0, x_1, \dots, x_n,\dots)$. We have $d(x_{-1}, x_n) = 1 + d(x_0, x_n)$ and $d(gx_{-1}, gx_n) = 1+ d(x_0, x_{n+\theta(g)})$ and therefore $\theta(g) = 0$. Now we show that $gx_0 = x_0$. If not, $(gx_0, gx_{-1}, x_0, x_1, \dots,)$ would be a half-line and with $n$ as above we would have 
$d(x_0, x_n) = d(gx_0, x_n) = 2 + d(x_0,x_n)$.

It follows that $G = HNN_\alpha(H)$ where $\alpha$ is the endomorphism $h \mapsto t^{-1}h t$, its Bass-Serre tree being $T_q$.  Note that $\cap_{n\geq 1} \alpha^n(H)$ has a trivial $H$-core since $G$ acts faithfully on $T_q$.
\end{proof}

\begin{rem}\label{rem:HNN} We keep the notation of the previous proof. It is easily checked that $\cup_{n\in \Z} G_n$ is the kernel of $\theta$ and so $G$ is the semi-direct product $(\cup_{n\in \Z} G_n) \rtimes \Z$ where $1\in \Z$ acts on $\cup_{n\in \Z} G_n$ as $g\mapsto tgt^{-1}$. 

Note that if $H$ is finitely generated, then $G$ has exponential growth \cite{BH}. More precisely, with respect to any finite generating set, we have $\lim_n a_n^{1/n} \geq 2^{1/4}$.
\end{rem}

\subsection{Construction of elements in $\cA_{T_q}$, $q\geq 2$}\label{sect:const_ex} 
Let $\mathcal{T}_q = (V,E)$ be  the $q$-regular  rooted tree: each vertex of level $n$ has $q$ vertices at level $n+1$ to which it is connected by an edge. We denote by $x_0$ the root of $\cT_q$. Let $\Lambda$ be a subgroup of the group $\Aut(\cT_q)$ of automorphisms of $\mathcal{T}_q$ (thus fixing $x_0$) that acts spherically transitively on $\cT_q$, meaning that the action  is transitive on each level. We consider another copy $\mathcal{T}'_q = (V',E')$ of $\mathcal{T}_q$. We denote by $x_1$ its root. We embed both rooted trees in the regular tree $T_q$ in the following way: the set of vertices of $T_q$ is $X = V\sqcup V'$, we keep the edges of each rooted tree and we add an edge between $x_0$ and $x_1$. We extend the action of $\Lambda$ to an action on $T_q$, so that its restriction to $\mathcal{T}'_q$ is trivial. We choose half-lines $(\dots x_{-n},x_{-n+1}, \dots x_{-1}, x_0)$ and $(x_1,x_2,\dots, x_n, \dots)$ of $\cT_q$ and $\cT'_q$ respectively and we consider the double infinite path of $T_q$,
$$\ell =(\dots, x_{-n},\dots, x_{-1}, x_0,x_1, x_2,\dots, x_n,\dots).$$

Now, let $t\in \Aut(T_q)$ be the translation (represented by the arrows in the figure 1 below, in case $q=2$)
 of step $1$ along the axis $\ell$ that translates the rooted subtree with root  $x_n$ (obtained after deleting the edges $(x_n, x_{n+1})$ and $(x_n, x_{n-1})$) to the corresponding subtree with root $x_{n+1}$.

Let $G$ be the subgroup of $\Aut(T_q)$ generated by $\set{t} \cup \Lambda$. Then $G$ acts amenably on $T_q$ since it fixes the end $\omega = [x_0,x_1,\dots, x_n,\dots]$. Moreover the action is transitive and faithful. If $H$ ($\supset \Lambda$) denotes the stabilizer of $x_0$ then $G = HNN_\alpha(H)$ where $\alpha: h\mapsto t^{-1}ht$.

For instance let $\Lambda$ be a group given with a decreasing sequence $ \Lambda = \Lambda_0 \supset \Lambda_1\supset\cdots \supset\Lambda_n\supset\cdots$ of subgroups such that $[\Lambda_n:\Lambda_{n+1}] = q$ for every $n$ and $\cap_n \Lambda_n = \set{e}$. The corresponding coset rooted tree is constructed as follows. Its root is $x_0 =\Lambda$, and for $n\geq 1$, its vertices at the $n$-th level are the cosets $\Lambda/\Lambda_n$. A vertex $g\Lambda_n$ is connected with a vertex $g'\Lambda_{n+1}$ if and only if $g'\Lambda_{n+1} \subset g\Lambda_n$. This rooted tree is  the $q$-regular  rooted tree $\mathcal{T}_p$. The action of $\Lambda$ on $\mathcal{T}_q$ is faithful, and spherically transitive.

\subsection{Examples of non-amenable groups in $\cA_{T_p}$} Let $p$ be a prime number and $\Lambda$ be resi\-dually a finite $p$-group. Then there is a sequence $ \Lambda \supset \Lambda_1\supset\dots \supset\Lambda_n\supset\dots$ of subgroups such that $[\Lambda_n:\Lambda_{n+1}] = p$ for every $n$ and $\cap_n \Lambda_n = \set{e}$ (see for instance \cite{Sidki}). It follows from the above construction that $\Lambda$ is contained in a group belonging to $\cA_{T_p}$.

Non-amenable residually  finite $p$-groups are plentiful. First, it is well-known that for every prime number $p$ and every integer $k\geq 2$, the free group $\F_k$ is  residually a finite $p$-group (see \cite{Val93} for an elementary proof).
Other examples are given by congruence subgroups. Denote by $\Gamma_n(k)$ the congruence subgroup of ${\rm SL}_n(\Z)$ defined as the kernel of the congruence homomorphism from ${\rm SL}_n(\Z)$ onto ${\rm SL}_n(\Z/k\Z)$. When $n\geq 3$, the group  $\Gamma_n(k)$ is residually $p$-finite if and only if the prime number $p$ divides $k$. For $n=2$, and $k\geq 3$, the group  $\Gamma_2(k)$ is free and therefore residually a finite $p$-group for every $p$, and for $k=2$, the group  $\Gamma_2(2)$ is residually $p$-finite only for $p=2$ (see \cite{Nica}).
In particular, the groups $\Gamma_n(p)$ for $p$ prime and $n\geq 3$ are residually finite $p$-groups and have Property (T).

\vspace{-1cm}

\begin{center}

{\begin{tikzpicture}[line cap=round,line join=round,>=triangle 45,x=0.6cm,y=1.0cm]
\clip(-11.589317506838734,-7.184963557963577) rectangle (19.544768826917874,6.8418043899891785);
\draw [line width=2.pt,color=ffqqqq] (0.,0.)-- (-5.,1.);

\draw (0.,0.)-- (5.,1.);
\draw [line width=2.pt,color=ffqqqq] (0.,-1.)-- (-5.,-2.);
\draw (0.,-1.)-- (5.,-2.);
\draw (7.,3.)-- (5.,1.);
\draw [line width=2.pt,color=ffqqqq] (-5.,1.)-- (-7.,3.);
\draw (-5.646983871941729,3.931881452539514)-- (-7.,3.);
\draw [line width=2.pt,color=ffqqqq] (-8.293050002867338,4.031451963488016)-- (-7.,3.);
\draw (2.9374107794433706,2.95880580482952)-- (5.,1.);
\draw (-4.26,3.98)-- (-3.,3.);
\draw (-1.74,3.98)-- (-3.,3.);
\draw (1.68,4.)-- (2.9374107794433706,2.95880580482952);
\draw (4.214901446007198,4.089699181787656)-- (2.9374107794433706,2.95880580482952);
\draw (5.953126451331749,4.068756711843987)-- (7.,3.);
\draw (8.36,3.9633759306802725)-- (7.,3.);
\draw [line width=2.pt,color=ffqqqq] (-7.,-4.)-- (-5.,-2.);
\draw (-2.884595864896693,-4.015036686412332)-- (-5.,-2.);
\draw (3.,-4.)-- (5.,-2.);
\draw (7.,-4.)-- (5.,-2.);
\draw [line width=2.pt,color=ffqqqq] (-8.21522452842835,-4.998591739334411)-- (-7.,-4.);
\draw (-5.7442657149904655,-4.998591739334411)-- (-7.,-4.);
\draw (-4.22666896343019,-4.998591739334411)-- (-2.884595864896693,-4.015036686412332);
\draw (-1.7557101499923065,-5.018048107944158)-- (-2.884595864896693,-4.015036686412332);
\draw (1.7658925683719215,-5.018048107944158)-- (3.,-4.);
\draw (4.217395013200058,-5.018048107944158)-- (3.,-4.);
\draw (7.,-4.)-- (5.851729976418816,-5.018048107944158);
\draw (7.,-4.)-- (8.264319684027457,-5.037504476553905);
\draw [line width=2.4pt,color=ffqqqq] (-8.74054648089152,-5.932497432602272)-- (-8.21522452842835,-4.998591739334411);
\draw (-7.787184419013912,-5.932497432602272)-- (-8.21522452842835,-4.998591739334411);
\draw (-6.172305824404901,-5.990866538431513)-- (-5.7442657149904655,-4.998591739334411);
\draw (-5.316225605576029,-5.990866538431513)-- (-5.7442657149904655,-4.998591739334411);
\draw (-4.693621810064121,-6.029779275651007)-- (-4.22666896343019,-4.998591739334411);
\draw (-3.720803379576765,-6.049235644260754)-- (-4.22666896343019,-4.998591739334411);
\draw (-2.2810321024554785,-5.990866538431513)-- (-1.7557101499923065,-5.018048107944158);
\draw (-1.2693009347486286,-5.990866538431513)-- (-1.7557101499923065,-5.018048107944158);
\draw (1.2211142472990022,-5.990866538431513)-- (1.7658925683719215,-5.018048107944158);
\draw (2.330127258054588,-6.029779275651007)-- (1.7658925683719215,-5.018048107944158);
\draw (3.692073060736886,-6.029779275651007)-- (4.217395013200058,-5.018048107944158);
\draw (4.684347859833989,-6.029779275651007)-- (4.217395013200058,-5.018048107944158);
\draw (5.423689867004379,-6.01032290704126)-- (5.851729976418816,-5.018048107944158);
\draw (6.415964666101482,-5.990866538431513)-- (5.851729976418816,-5.018048107944158);
\draw (7.816823206003274,-5.990866538431513)-- (8.264319684027457,-5.037504476553905);
\draw (8.750728899271136,-6.01032290704126)-- (8.264319684027457,-5.037504476553905);
\draw [line width=2.pt,color=ffqqqq] (-8.847683897202629,4.9709555386534205)-- (-8.293050002867338,4.031451963488016);
\draw (-7.7747603883941006,4.988837597133563)-- (-8.293050002867338,4.031451963488016);
\draw (-6.183257183661451,4.988837597133563)-- (-5.646983871941729,3.931881452539514);
\draw (-5.181861908773491,5.006719655613705)-- (-5.646983871941729,3.931881452539514);
\draw (-4.591753978928801,4.988837597133563)-- (-4.26,3.98);
\draw (-3.6797689964415516,4.988837597133563)-- (-4.26,3.98);
\draw (-2.231322259550039,5.006719655613705)-- (-1.74,3.98);
\draw (-1.31933727706279,5.006719655613705)-- (-1.74,3.98);
\draw (1.1841509101571086,4.988837597133563)-- (1.68,4.);
\draw (2.292838535925921,5.006719655613705)-- (1.6599201128795436,4.026871771956649);
\draw (3.651874980416723,5.006719655613705)-- (4.214901446007198,4.089699181787656);
\draw (4.706916430745109,5.006719655613705)-- (4.214901446007198,4.089699181787656);
\draw (5.440080828430936,5.006719655613705)-- (5.953126451331749,4.068756711843987);
\draw (6.441476103318896,4.953073480173279)-- (5.953126451331749,4.068756711843987);
\draw (7.746866372369272,4.9709555386534205)-- (8.36,3.9633759306802725);
\draw (8.712497530296947,4.988837597133563)-- (8.36,3.9633759306802725);
\draw (-3.,3.)-- (-5.,1.);
\draw [line width=2.pt,color=ffqqqq] (0.,0.)-- (0.,-1.);
\draw [dotted] (-7.,3.)-- (7.,3.);
\draw [dotted] (-5.,1.)-- (5.,1.);
\draw [dotted] (-8.293050002867338,4.031451963488016)-- (8.36,3.9633759306802725);
\draw [dotted] (-8.847683897202629,4.9709555386534205)-- (8.712497530296947,4.988837597133563);
\draw [dotted] (-5.,-2.)-- (5.,-2.);
\draw [dotted] (-7.,-4.)-- (7.,-4.);
\draw [dotted] (-8.21522452842835,-4.998591739334411)-- (8.264319684027457,-5.037504476553905);
\draw [dotted] (-8.74054648089152,-5.932497432602272)-- (8.750728899271136,-6.01032290704126);
\draw [->] (-9.984093175800584,-5.9208014512862315) -- (-9.48147389715252,-5.1459300633704705);
\draw [->] (-4.75335716991852,-1.6334992005501785) -- (-1.2284047010739094,-0.892458056531713);
\draw [->] (-8.748487449124095,-4.62236831477874) -- (-7.764191361771639,-3.868439396806648);
\draw [->] (-7.198744673292568,-3.4286475279895945) -- (-5.62805942751737,-1.9626746319327488);
\draw [->] (-0.5390392312057323,-0.7480113751999341) -- (-0.5571204461129169,-0.25953631569393976);
\draw [->] (-0.9205975591188162,-0.16866703744246528) -- (-4.668955286992153,0.6264391472579361);
\draw [->] (-5.795599187066725,1.0111561000682818) -- (-7.261572083123576,2.539956405956135);
\draw [->] (-7.659479012053292,2.9378633348858503) -- (-8.497177809800064,3.649907312970604);
\draw [->] (-8.811314858955104,4.0268717719566505) -- (-9.272049197715829,4.759858219985073);
\draw [shift={(-2.734854389132489,4.848166573469868)},line width=0.4pt]  plot[domain=-3.1364380600672526:0.0051545935225405515,variable=\t]({1.*2.031446572075657*cos(\t r)+0.*2.031446572075657*sin(\t r)},{0.*2.031446572075657*cos(\t r)+1.*2.031446572075657*sin(\t r)});
\draw [->,line width=0.4pt] (8.135015260564087,-1.310658210015226) -- (8.12930134524493,1.5305871607160741);
\draw [->,line width=0.4pt] (1.0688824393629077,2.2637014992060793) -- (-1.9212422474879056,2.7493811053928274);
\draw [color=ffqqqq](-9.343123498765864,5.969939211089099) node[anchor=north west] {$\omega$};
\draw [shift={(-5.742990604991872,4.53194857101934)},line width=0.4pt]  plot[domain=3.094009550312811:6.235602203902604,variable=\t]({1.*0.7974276582539525*cos(\t r)+0.*0.7974276582539525*sin(\t r)},{0.*0.7974276582539525*cos(\t r)+1.*0.7974276582539525*sin(\t r)});
\draw [->,line width=0.4pt] (-4.3775190393288685,3.261301419638496) -- (-5.136114353586093,3.77335325676212);
\draw [shift={(5.000615533176064,-6.069420945725321)},line width=0.4pt]  plot[domain=0.020961290086983032:3.1625539436767762,variable=\t]({1.*4.524118418676244*cos(\t r)+0.*4.524118418676244*sin(\t r)},{0.*4.524118418676244*cos(\t r)+1.*4.524118418676244*sin(\t r)});
\draw [shift={(4.9152735603221265,5.0724477324276105)},line width=0.4pt]  plot[domain=3.130723516412523:6.272316170002316,variable=\t]({1.*4.362180724407386*cos(\t r)+0.*4.362180724407386*sin(\t r)},{0.*4.362180724407386*cos(\t r)+1.*4.362180724407386*sin(\t r)});
\draw (-10.346065523630255,-5.918593083644906) node[anchor=north west] {$x_{-4}$};
\draw (-9.489895502404556,-4.597645050896682) node[anchor=north west] {$x_{-3}$};
\draw (-7.924327463591847,-3.4234690217871515) node[anchor=north west] {$x_{-2}$};
\draw (-5.9673674150759615,-1.39312297145192) node[anchor=north west] {$x_{-1}$};
\draw (-0.2677212737734446,-1.0506549629616402) node[anchor=north west] {$x_0$};
\draw (-0.34110727559279036,0.6372230788833112) node[anchor=north west] {$x_1$};
\draw (-5.64936140719213,0.8818430849477968) node[anchor=north west] {$x_2$};
\draw (-8.340181473901472,2.936651135889477) node[anchor=north west] {$x_3$};
\draw (-10.297141522417359,5.28500319410854) node[anchor=north west] {$x_5$};
\draw (-0.6591132834766218,-1.7355909799422) node[anchor=north west] {$\mathcal T_2$};
\draw (-0.6591132834766218,1.7135511055670483) node[anchor=north west] {$\mathcal T_{2}'$};
\draw [dotted] (-8.847683897202629,4.9709555386534205)-- (-9.154213304299802,5.5317412425582155);
\draw (-9.465433501798106,3.939593160753868) node[anchor=north west] {$x_4$};
\draw [->,line width=0.4pt] (-1.1631528059724934,-3.7260639692889086) -- (1.436179378468733,-2.934962869676363);
\draw [shift={(-2.632340562395796,-6.113494073476769)},line width=0.4pt]  plot[domain=0.018290642956820036:3.1598832965466133,variable=\t]({1.*2.317183670771681*cos(\t r)+0.*2.317183670771681*sin(\t r)},{0.*2.317183670771681*cos(\t r)+1.*2.317183670771681*sin(\t r)});
\begin{scriptsize}
\draw [fill=black] (0.,0.) circle (2.5pt);
\draw [fill=black] (-5.,1.) circle (2.5pt);
\draw [fill=black] (5.,1.) circle (2.5pt);
\draw [fill=black] (0.,-1.) circle (2.5pt);
\draw [fill=black] (-5.,-2.) circle (2.5pt);
\draw [fill=black] (5.,-2.) circle (2.5pt);
\draw [fill=black] (-7.,3.) circle (2.5pt);
\draw [fill=black] (7.,3.) circle (2.5pt);
\draw [fill=black] (-5.646983871941729,3.931881452539514) circle (2.5pt);
\draw [fill=black] (-8.293050002867338,4.031451963488016) circle (2.5pt);
\draw [fill=black] (-3.,3.) circle (2.5pt);
\draw [fill=black] (2.9374107794433706,2.95880580482952) circle (2.5pt);
\draw [fill=black] (-4.26,3.98) circle (2.5pt);
\draw [fill=black] (-1.74,3.98) circle (2.5pt);
\draw [fill=qqqqff] (1.68,4.) circle (2.5pt);
\draw [fill=black] (4.214901446007198,4.089699181787656) circle (2.5pt);
\draw [fill=black] (5.953126451331749,4.068756711843987) circle (2.5pt);
\draw [fill=black] (8.36,3.9633759306802725) circle (2.5pt);
\draw [fill=black] (-7.,-4.) circle (2.5pt);
\draw [fill=black] (-2.884595864896693,-4.015036686412332) circle (2.5pt);
\draw [fill=black] (3.,-4.) circle (2.5pt);
\draw [fill=black] (7.,-4.) circle (2.5pt);
\draw [fill=black] (-8.21522452842835,-4.998591739334411) circle (2.5pt);
\draw [fill=black] (-5.7442657149904655,-4.998591739334411) circle (2.5pt);
\draw [fill=black] (-4.22666896343019,-4.998591739334411) circle (2.5pt);
\draw [fill=black] (-1.7557101499923065,-5.018048107944158) circle (2.5pt);
\draw [fill=black] (1.7658925683719215,-5.018048107944158) circle (2.5pt);
\draw [fill=black] (4.217395013200058,-5.018048107944158) circle (2.5pt);
\draw [fill=black] (5.851729976418816,-5.018048107944158) circle (2.5pt);
\draw [fill=black] (8.264319684027457,-5.037504476553905) circle (2.5pt);
\draw [fill=black] (-8.74054648089152,-5.932497432602272) circle (2.5pt);
\draw [fill=black] (-7.787184419013912,-5.932497432602272) circle (2.5pt);
\draw [fill=black] (-6.172305824404901,-5.990866538431513) circle (2.5pt);
\draw [fill=black] (-5.316225605576029,-5.990866538431513) circle (2.5pt);
\draw [fill=black] (-4.693621810064121,-6.029779275651007) circle (2.5pt);
\draw [fill=black] (-3.720803379576765,-6.049235644260754) circle (2.5pt);
\draw [fill=black] (-2.2810321024554785,-5.990866538431513) circle (2.5pt);
\draw [fill=black] (-1.2693009347486286,-5.990866538431513) circle (2.5pt);
\draw [fill=black] (1.2211142472990022,-5.990866538431513) circle (2.5pt);
\draw [fill=black] (2.330127258054588,-6.029779275651007) circle (2.5pt);
\draw [fill=black] (3.692073060736886,-6.029779275651007) circle (2.5pt);
\draw [fill=black] (4.684347859833989,-6.029779275651007) circle (2.5pt);
\draw [fill=black] (5.423689867004379,-6.01032290704126) circle (2.5pt);
\draw [fill=black] (6.415964666101482,-5.990866538431513) circle (2.5pt);
\draw [fill=black] (7.816823206003274,-5.990866538431513) circle (2.5pt);
\draw [fill=black] (8.750728899271136,-6.01032290704126) circle (2.5pt);
\draw [fill=black] (-8.847683897202629,4.9709555386534205) circle (2.5pt);
\draw [fill=black] (-7.7747603883941006,4.988837597133563) circle (2.5pt);
\draw [fill=black] (-6.183257183661451,4.988837597133563) circle (2.5pt);
\draw [fill=black] (-5.181861908773491,5.006719655613705) circle (2.5pt);
\draw [fill=black] (-4.591753978928801,4.988837597133563) circle (2.5pt);
\draw [fill=black] (-3.6797689964415516,4.988837597133563) circle (2.5pt);
\draw [fill=black] (-2.231322259550039,5.006719655613705) circle (2.5pt);
\draw [fill=black] (-1.31933727706279,5.006719655613705) circle (2.5pt);
\draw [fill=black] (1.1841509101571086,4.988837597133563) circle (2.5pt);
\draw [fill=black] (2.292838535925921,5.006719655613705) circle (2.5pt);
\draw [fill=black] (1.6599201128795436,4.026871771956649) circle (2.5pt);
\draw [fill=black] (3.651874980416723,5.006719655613705) circle (2.5pt);
\draw [fill=black] (4.706916430745109,5.006719655613705) circle (2.5pt);
\draw [fill=black] (5.440080828430936,5.006719655613705) circle (2.5pt);
\draw [fill=black] (6.441476103318896,4.953073480173279) circle (2.5pt);
\draw [fill=black] (7.746866372369272,4.9709555386534205) circle (2.5pt);
\draw [fill=black] (8.712497530296947,4.988837597133563) circle (2.5pt);
\draw [fill=black] (-9.984093175800584,-5.9208014512862315) circle (0.5pt);
\draw [fill=black] (-9.48147389715252,-5.1459300633704705) circle (0.5pt);
\draw [fill=black] (-4.75335716991852,-1.6334992005501785) circle (0.5pt);
\draw [fill=black] (-1.2284047010739092,-0.892458056531713) circle (0.5pt);
\draw [fill=black] (-8.748487449124095,-4.62236831477874) circle (0.5pt);
\draw [fill=black] (-7.764191361771639,-3.868439396806648) circle (0.5pt);
\draw [fill=black] (-7.198744673292568,-3.4286475279895945) circle (0.5pt);
\draw [fill=black] (-5.62805942751737,-1.9626746319327488) circle (0.5pt);
\draw [fill=black] (-0.5390392312057323,-0.7480113751999341) circle (0.5pt);
\draw [fill=black] (-0.5571204461129169,-0.25953631569393976) circle (0.5pt);
\draw [fill=black] (-0.9205975591188162,-0.16866703744246528) circle (0.5pt);
\draw [fill=black] (-4.668955286992153,0.6264391472579361) circle (0.5pt);
\draw [fill=black] (-5.795599187066725,1.0111561000682818) circle (0.5pt);
\draw [fill=black] (-7.261572083123576,2.539956405956135) circle (0.5pt);
\draw [fill=black] (-7.659479012053292,2.9378633348858503) circle (0.5pt);
\draw [fill=black] (-8.497177809800064,3.649907312970604) circle (0.5pt);
\draw [fill=black] (-8.811314858955104,4.0268717719566505) circle (0.5pt);
\draw [fill=black] (-9.272049197715829,4.759858219985073) circle (0.5pt);
\draw [fill=black] (-0.7034348045965657,4.858637808441703) circle (0.5pt);
\draw [fill=black] (-4.766273973668412,4.837695338498033) circle (0.5pt);
\draw [fill=black] (8.135015260564087,-1.310658210015226) circle (0.5pt);
\draw [fill=black] (8.12930134524493,1.5305871607160741) circle (0.5pt);
\draw [fill=black] (1.0688824393629077,2.2637014992060793) circle (0.5pt);
\draw [fill=black] (-1.9212422474879056,2.7493811053928274) circle (0.5pt);
\draw [fill=black] (-4.946465525021787,4.49401880530648) circle (0.5pt);
\draw [fill=black] (-6.539515684961957,4.5698783367322005) circle (0.5pt);
\draw [fill=black] (-4.3775190393288685,3.261301419638496) circle (0.5pt);
\draw [fill=black] (-5.136114353586093,3.77335325676212) circle (0.5pt);
\draw [fill=black] (0.4774909719173654,-6.164245360007474) circle (0.5pt);
\draw [fill=black] (9.523740094434762,-5.974596531443169) circle (0.5pt);
\draw [fill=black] (0.5533505033430878,5.119859939568687) circle (0.5pt);
\draw [fill=black] (9.277196617301165,5.025035525286534) circle (0.5pt);
\draw [fill=black] (-9.154213304299802,5.5317412425582155) circle (0.5pt);
\draw [fill=black] (-1.1631528059724934,-3.7260639692889086) circle (0.5pt);
\draw [fill=black] (1.4361793784687331,-2.934962869676363) circle (0.5pt);
\draw [fill=black] (-4.949136639832542,-6.155874489527441) circle (0.5pt);
\draw [fill=black] (-0.3155444849590495,-6.071113657426097) circle (0.5pt);
\end{scriptsize}
\end{tikzpicture}}
\end{center}

\vspace{-0.5cm}
\hspace{6cm}{\bf Figure 1}

  \subsection{Examples of amenable groups in $\cA_{T_q}$} They are the groups of the form  $G = HNN_\alpha(H)$, where $H$ is an amenable group with  $[H:\alpha(H)] = q$ and such that $\cap_{n\geq 1} \alpha^n(H)$ has a trivial $H$-core.

 For instance, take $\alpha : x\mapsto qx$ from $H=\Z$ onto $K=q\Z$. Then $G = HNN_\alpha(H)$ is  the Baumslag-Solitar group $G=BS(1,q) = \scal{a,t| t^{-1}at = a^q}$, $q\geq2$, which acts transitively on its Bass-Serre tree $T_{q} = (X,E)$.  The end $[x_0=H,x_1= tH, \dots,x_k = t^{k}H,\dots]$ is $G$-invariant. This group is isomorphic to the solvable group of affine transformations of $\R$ generated by $a: x\mapsto x+1$ and $t: x\mapsto q^{-1}x$, that is the group of affine transformations of the form $x\mapsto \lambda x + \mu$ with $\lambda \in q^\Z$ and $\mu\in \Z[1/q]= \set{mq^k: m,k\in \Z}$. The stabilizer $G_k$ of $x_k$ is the group of translations by elements of $\Z/q^k$. Therefore $\cup_{k>0} G_k = \Z[1/q]$. 
  
For another example of amenable group, start with the integer Heisenberg group $H =  \Z^3$ with the product $(x,z,y)(x',z',y') = (x+x', z+z' + xy', y+y')$ and take $\alpha: (x,z,y) \mapsto (2x,4z,2y)$.

\vspace{5mm}

\noindent{\bf Aknowledgements.} The author would like to express her thanks to N. Monod for an email exchange leading to the examples given in Section \ref{sect:const_ex}, as well as to Y. de Cornulier, P. de la Harpe and A. Valette for their critical reading of a preliminary version of this paper  and for useful bibliographical informations.

 \bibliographystyle{plain}

\end{document}